\documentclass[]{interact}

\usepackage{epstopdf}% To incorporate .eps illustrations using PDFLaTeX, etc.
\usepackage[caption=false]{subfig}% Support for small, `sub' figures and tables

\usepackage{algorithm}

\usepackage[numbers,sort&compress]{natbib}% Citation support using natbib.sty
\bibpunct[, ]{[}{]}{,}{n}{,}{,}% Citation support using natbib.sty
\usepackage{xcolor}
\usepackage{soul}

\theoremstyle{plain}% Theorem-like structures provided by amsthm.sty
\newtheorem{theorem}{Theorem}[section]
\newtheorem{lemma}[theorem]{Lemma}

\theoremstyle{definition}
\newtheorem{definition}[theorem]{Definition}

\theoremstyle{remark}

\newcommand{\R}{\mathbb{R}}
\newcommand{\N}{\mathbb{N}}
\newcommand{\E}{\mathbb{E}}
\newcommand{\K}{\mathcal{K}}

\newcommand{\pol}{^{\textnormal{o}}}

\newcommand{\pcon}[1]{\Pi_{\K}(#1)}
\newcommand{\ppol}[1]{\Pi_{\K^\textnormal{o}}(#1)}

\newcommand{\bproj}[2]{\Pi_{#1}\left(#2\right)}
\newcommand{\proj}[2]{\Pi_{#1}(#2)}
\newcommand{\seq}[1]{\{#1^k\}_{k\in \N}}

\newcommand{\xb}{\bar{x}}

\renewcommand{\Im}{\textnormal{Im}}
\newcommand{\Ker}{\textnormal{Ker}}

\renewcommand{\seq}[1]{\{#1^k\}_{k\in \N}}

\def\CC{\mbox{$\cal{C}$}}
\def\KC{\mbox{$\cal{K}$}}
\def\SC{\mbox{$\cal{S}$}}
\def\NC{\mbox{$\cal{N}$}}
\def\OC{\mbox{$\cal{O}$}}

\begin{document}

%\articletype{ARTICLE TEMPLATE}% Specify the article type or omit as appropriate

\title{Exploiting cone approximations in an augmented Lagrangian method for conic optimization}

\author{
\name{Mituhiro Fukuda\textsuperscript{a,b,c,d}\thanks{CONTACT M.~F. Email: mituhiro@usp.br}, Walter G\'omez\textsuperscript{e}, Gabriel Haeser\textsuperscript{f}, and Leonardo M. Mito\textsuperscript{f}}
\affil{\textsuperscript{a}School of Arts, Sciences and Humanities, University of S\~ao Paulo, S\~ao Paulo, SP, Brazil; \textsuperscript{b}Center for Mathematics, Computing and Cognition, Federal University of ABC, Santo Andr\'e, SP, Brazil; \textsuperscript{c}Department of Computer Science, University of S\~ao Paulo, S\~ao Paulo, SP, Brazil; \textsuperscript{d}Department of Mathematical and Computing Science, Institute of Science Tokyo, Meguro-ku, Tokyo, Japan; \textsuperscript{e}Department of Mathematical Engineering, Universidad de la Frontera, Temuco, Chile; \textsuperscript{f}Department of Applied Mathematics, University of S\~ao Paulo, S\~ao Paulo, SP, Brazil.}
}
\date{May 27, 2024, revised December 1, 2024, and April 18, 2025}

\maketitle

\begin{abstract}
We propose an algorithm for general nonlinear conic programming which does not require the knowledge of the full cone, but rather a simpler, more tractable, approximation of it. We prove that the algorithm satisfies a strong global convergence property in the sense that it generates a strong sequential optimality condition. In particular, a KKT point is necessarily found when a limit point satisfies Robinson's condition. We conduct numerical experiments minimizing nonlinear functions subject to a copositive cone constraint. In order to do this, we consider a well known polyhedral approximation of this cone by means of refining the polyhedral constraints after each augmented Lagrangian iteration. We show that our strategy outperforms the standard approach of considering a close polyhedral approximation of the full copositive cone in every iteration.
\end{abstract}


\begin{amscode}
90C30; 49M37; 90C26
\end{amscode}

\section{Introduction}

We study a very general class of optimization problems, sometimes referred to as \textit{nonlinear conic programming} (NCP), which can be stated in the form:
\begin{equation}\label{prob:ncp}
	\begin{aligned}
	& \underset{x\in \R^n}{\text{Minimize}}
	& & f(x), \\
	& \text{subject to}
	& & g(x)\in \K,\\
	\end{aligned}
	\tag{NCP}
\end{equation}
where $f\colon\R^n\to\R$ and $g\colon\R^n\to \E$ are continuously differentiable, $\E$ is a finite-dimensional vector space equipped with an inner product $\langle\cdot,\cdot\rangle$, and $\K\subset \E$ is a nonempty closed convex cone. We will denote the feasible set of~(NCP) by $\Omega$, which is assumed to be nonempty, and the Lagrangian function of (NCP) is defined as $L(x,\mu)=f(x)+\langle g(x),\mu\rangle,\textrm{ for } \mu\in \E$.

Despite its generality and simplicity of representation, studies on practical methods for solving~(NCP) and even some of its particular cases are somewhat rare in the current literature, in comparison with the cases when $\K$ is polyhedral, i.e., \textit{nonlinear programming} (NLP); or when $g$ is affine, i.e., \textit{conic programming} (CP). This is no surprise, in view of the fact that evaluating feasibility in (N)CP may be computationally expensive or even NP-Hard, such as in the particular case of \textit{copositive programming} (COP). For such intricate problems, some authors have found sequences of polyhedral approximations of $\K$ and used them to compute bounds for the solutions of the original problems -- see, for instance, \cite{duropen,dursurvey,dursurvey2} and references therein for additional information on COP. A particularly interesting algorithmic approach for solving COP problems was presented in Bundfuss and D{\"u}r~\cite{duradapt} and also in Y{\i}ld{\i}r{\i}m~\cite{alper}, which consists of computing increasingly better polyhedral approximations of $\K$ and solving the approximate problems induced by them to generate a sequence of approximate solutions for the original problem. Both works~\cite{duradapt,alper} present impressive numerical results, using distinct approximations of $\K$.

Recently, Andreani et al.~\cite{Andreani2020} extended the so-called ``sequential optimality conditions'' from NLP to the NCP world. In summary, sequential optimality conditions are parametric perturbed forms of the traditional Karush-Kuhn-Tucker (KKT) conditions with three distinguishing properties:
\begin{enumerate}
\item They are always fulfilled at local minimizers, even for degenerate problems;
\item They are equivalent to the standard KKT condition in the presence of a constraint qualification;
\item They are naturally satisfied at all the feasible accumulation points computed by a practical algorithm such as \cite{amhaeser, ahv, amsvaiter}.
\end{enumerate}

The interested reader may check the earlier texts on sequential conditions in NLP~\cite{amhaeser,amsvaiter} and an extension to \textit{nonlinear semidefinite programming} (NSDP)~\cite{ahv} for more details. All these works use an augmented Lagrangian method to illustrate item 3 of the list above, but one of the most critical disadvantages of it in the conic context is the need of a projection onto $\K$ that must be computed several times per iteration and that may be numerically very expensive. In this paper, we will study the effects of replacing $\K$ with increasingly better polyhedral approximations of it, and the novelty in our proposal is the possibility of improving %said (?)
approximation mid-execution. As far as we are concerned, there is no previous work that allows this in the literature. Our first concern is to build a solid convergence theory for our method, which is done by means of a modified variant of a sequential optimality condition called \textit{approximate gradient projection} (AGP), first introduced in~\cite{mariosvaiter} for NLP and then extended to the conic framework in~\cite{Andreani2020}.

To illustrate the behavior of our method, we use it for solving a nonlinear COP problem via Y{\i}ld{\i}r{\i}m's approximations~\cite{alper}.

Since this is a first approach, there are some limitations. We can only prove that feasible
limit points will satisfy the necessary optimality condition under Robinson's constraint qualification. Therefore,
the algorithm can generate points which might not get close to optimal if the Lagrange multipliers are not bounded.
Also a careful polyhedral approximation of the closed convex cone $\K$ is necessary to guarantee cheap computation without turning
the subproblems too complex to solve.

There are many applications of the (NCP) formulation for tractable cones, for instance in case of nonlinear semidefinite programming \cite{birgin20,Stingl2009}.
The copositive cone in contrast is an instance of an intractable cone.
As shown for example in \cite{duropen, duradapt, dursurvey}, NP-hard binary quadratic
programming problems can be equivalently reformulated as linear programs
over the copositive cone. The investigation of an equivalent reformulation
of nonlinear binary optimization problems as nonlinear copositve
programs would be an interesting extension. To our knowledge, this
has not been considered yet, possibly due to the lack of algorithms to deal with it. We believe that with our contribution it will be possible to model difficult problems  with nonconvex (NCP) and to find approximate solutions using the proposed algorithm.
%  We were not able to find straightforward applications of the current formulation (NCP) in the
%literature. However, we believe that this field is in a maturity where this article will permit to model
%more complex problems that can be actually solved by the proposed algorithm.
% \cite{duropen,dursurvey,dursurvey2}.

%\textcolor{blue}{\st{To illustrate the behavior of our method, we use it for solving a nonlinear COP problem via Y{\i}ld{\i}r{\i}m's approximations} \cite{alper}.}

\subsection{Notations}

We will adopt the following notations.

\begin{itemize}
\item The natural number set is defined as $\N =\{0,1,\ldots,\}$;
\item $\R_{++}$ denotes the set of real positive numbers;
\item $\|z\|$ is the norm defined by the inner product in $\E\ni z$, if it is not otherwise specified;
\item $\|y\|_2$, $\|y\|_{\infty}$, and $\|y\|$ are the Euclidean, infinity, and any norm in $\R^n\ni y$, respectively;
\item $\|X\|_F$ is the Frobenius norm for an $m\times m$ real symmetric matrix $X$;
\item $B[\bar{x},\delta]$ is the closed ball centered at $\bar{x}$ with radius $\delta\geq 0$;
\item $\proj{\K}{z}$ denotes the orthogonal projection of $z$ onto $\K$ by the inner product $\langle\cdot,\cdot\rangle$ in $\E$;
\item $\K\pol$ and $\K^*$ stand for the polar cone and dual cone of a nonempty closed convex cone $\K\subseteq \E$, respectively;
\item $Dg(x)^*$ stands for the Jacobian of $g\colon\R^n\rightarrow \E$;
\item $g^{-1}\colon\E\rightarrow\R^n$ denotes the inverse image of $g\colon\R^n\rightarrow \E$.
\end{itemize}

\section{General framework}

We study situations where the cone $\K$ can be approximated by a sequence $\seq{\K}$ that converges to $\K$ in some sense, such that projecting onto each $\K^k$ is relatively easy.

\begin{definition}[Continuous approximation of $\K$]\label{def:goodapprox}
A sequence $\seq{\K}$ of nonempty closed convex cones is a continuous approximation for $\K$ when:
\begin{itemize}
\item For every sequence $\seq{y}\to y$ with $y^k\in \K^k$ for all $k\in \N$, we must have $y\in \K$. In other words, $\limsup_{k\in \N}\K^k\subseteq \K$.
\item For every $y\in \K$, there exists a sequence $\seq{y}\to y$ with $y^k\in \K^k$, for all $k\in \N$. In other words, $\K\subseteq \liminf_{k\in \N}\K^k$.
\end{itemize}
\end{definition}

When Definition \ref{def:goodapprox} holds, we denote it by $\lim_{k\in \N}\K^k= \K$. Let us begin with a technical lemma regarding continuous cone approximations, which is simply a generalization of some classical results on projections:

\begin{lemma}\label{lem:projseq}
Let $\K$ and $\seq{\K}$ be nonempty closed convex cones in $\E$. Then:
\begin{enumerate}
\item If $\K\subseteq \liminf_{k\in \N}\K^k$, then $\proj{\K^k}{z}\to z$ for every $z\in \K$, where
$\proj{\K^k}{z}$ is the orthogonal projection of $z$ onto $\K^k$;

\item If $\lim_{k\in \N}\K^k= \K$, then $\proj{\K^k}{z^k}\to \pcon{z}$ for every converging sequence $\seq{z}\to z\in \E$.
\end{enumerate}
\end{lemma}

\begin{proof}

\

\begin{enumerate}

\item Let $z\in \K$. By the hypothesis, there exists a sequence $\seq{z}\to z$ such that $z^k\in \K^k$ for each $k\in \N$. From the definition of projection, $\|\proj{\K^k}{z}-z\|\leqslant \|z^k-z\|$, whence follows the result;

\item Since the projection is nonexpansive and $0\in\K^k$, $\{\proj{\K^k}{z}\}_{k\in \N}$ is bounded by $\|z\|$. Let $N\subseteq \N$ be any infinite subset of $\N$ such that $\{\proj{\K^k}{z}\}_{k\in N}$ converges to, say, $w\in\E$. Since $\limsup_{k\in \N}\K^k\subseteq \K$ we obtain $w\in\K$. Now, for every $y\in \K$, it follows from the previous item that 
$$\|z-w\|
=
\lim_{k\in N} \|z-\proj{\K^k}{z}\|
\leqslant 
\lim_{k\in N} \|z-\proj{\K^k}{y}\|
=
\|z-y\|,$$ 
which means $w=\pcon{z}$, hence $\lim_{k\in\N}\proj{\K^k}{z}=\pcon{z}$. Now,
\begin{equation*}
\begin{array}{ll}
\|\proj{\K^k}{z^k} -\pcon{z}\| & \leqslant \|\proj{\K^k}{z^k}-\proj{\K^k}{z}\| + \|\proj{\K^k}{z}-\pcon{z}\|\\
              & \leqslant \|z^k-z\|+\|\proj{\K^k}{z}-\pcon{z}\|\to0,
\end{array}
\end{equation*}
which completes the proof.
\end{enumerate}
\end{proof}

To study global convergence of algorithms that benefit from continuous approximations of $\K$, we propose an adapted version of the sequential optimality condition from~\cite{Andreani2020}, called \emph{approximate gradient projection} (AGP):

\begin{definition}[R-AGP]\label{def:ragp}
Let $\seq{\K}$ be a continuous approximation of $\K$, in the sense of Definition \ref{def:goodapprox}. We say a feasible point $\bar{x}$ of (NCP) satisfies the \textit{relaxed cone AGP} (R-AGP) condition if there exist sequences $\seq{x}\rightarrow \bar{x}$ and $\seq{\mu}$ such that $\mu^k\in(\mathcal{K}^k)\pol$ for every $k\in \N$ and
\begin{enumerate}
\item $\nabla L(x^k,\mu^k):=\nabla f(x^k)+Dg(x^k)^{*}[\mu^k]\to 0,$
\item $\langle \mu^{k}, \proj{\K^k}{g(x^{k})} \rangle \rightarrow 0$.
\end{enumerate}
%\textcolor{blue}{\st{where $\proj{\K^k}{g(x^{k})}$ is the orthogonal projection of $g(x^{k})$ onto $\K^k$.}}
\end{definition}

The reason of requiring $\mu^k$ to be in the polar cone $(\K^k)^{\pol}$ instead of in the
dual cone of $\K^k$ will be clear in the proof of Theorem~\ref{thm:ragpnec}.
Recall that the KKT condition for (NCP) at $(\bar{x},\bar{\mu})\in \R^n\times \K\pol$ is
\begin{enumerate}
\item $\nabla L(\bar{x},\bar{\mu}) = \nabla f(\bar{x})+ Dg(\bar{x})^*[\bar{\mu}] = 0$;
\item $\langle \bar{\mu}, \proj{\K}{g(\bar{x})}\rangle = 0$;
\item $\proj{\K}{g(\bar{x})}=g(\bar{x})$.
\end{enumerate}

Next, we prove that R-AGP is a necessary optimality condition. For this purpose, we concisely prove the convergence of a variation of the external penalty algorithm. This algorithm, even though it seems novel, is not highlighted here because R-AGP will be used later for building the convergence theory of an augmented Lagrangian algorithm, which is where we intend to focus. For doing so, besides continuity of the cone approximation, we assume also that locally the original feasible set $g^{-1}(\K)$ is asymptotically included in the approximate feasible set $g^{-1}(\K^k)$. The proof employs an auxiliary problem which is an unconstrained version of the problem with two additional terms: a penalizing one to enforce feasibility and another regularizing one as in \cite{Andreani2020, amhaeser, ahv, amsvaiter}.

\begin{theorem}\label{thm:ragpnec}
If $\bar{x}$ is a local minimizer of (NCP), $\seq{\K}$ is a continuous approximation of $\K$, and there is a $\delta>0$ such that $g^{-1}(\K)\cap B[\bar{x},\delta]\subseteq \liminf_{k\in \N} g^{-1}(\K^k)\cap B[\bar{x},\delta]$, then $\xb$ satisfies R-AGP.
\end{theorem}

\begin{proof}
Let us assume that $\delta>0$ is small enough such that $f(\bar{x})\leq f(x)$ for all $x\in g^{-1}(\K)\cap B[\bar{x},\delta]$ and let $\{\rho_k\}_{k\in \N}\to\infty$. Consider the auxiliary problem of (NCP):
\begin{equation}\label{prob:penalty}
\begin{array}{ll}
\underset{x\in \R^n}{\text{Minimize}} & P_k(x)\doteq f(x)+\frac{\rho_k}{2}\|\proj{(\K^k)\pol}{g(x)}\|^2+\frac{1}{2}\|x-\bar{x}\|^2\\
\text{subject to} & \|x-\bar{x}\|\leqslant \delta.
\end{array}
\tag{RegA-k}
\end{equation}
Let $\seq{x}$ be a sequence of global minimizers of \eqref{prob:penalty}, which is bounded by the ball centered at $\xb$ with radius $\delta$, and let $w$ be an arbitrary limit point of it. Then, by the optimality of $x^k$, we have $P_k(x^k)\leqslant P_k(\xb)$ for every $k\in\N$, which implies
\begin{equation}\label{eqn:eq0}
\frac{\rho_k}{2}\|\proj{(\K^k)\pol}{g(\bar{x})}\|^2\geqslant f(x^k)-f(\bar{x})+\frac{\rho_k}{2}\|\proj{(\K^k)\pol}{g(x^k)}\|^2+\frac{1}{2}\|x^k-\bar{x}\|^2.
\end{equation}
Dividing everything by $\frac{\rho_k}{2}$, we obtain
\begin{equation}\label{eqn:eq1}
\|\proj{(\K^k)\pol}{g(\bar{x})}\|^2\geqslant \frac{2(f(x^k)-f(\bar{x}))}{\rho_k}+\|\proj{(\K^k)\pol}{g(x^k)}\|^2+\frac{\|x^k-\bar{x}\|^2}{\rho_k}\geqslant \frac{2(f(x^k)-f(\bar{x}))}{\rho_k},
\end{equation}
for every $k\in \N$. Since $\proj{(\K^k)\pol}{g(\bar{x})}=g(\bar{x})-\proj{\K^k}{g(\bar{x})}\to0$ (Lemma \ref{lem:projseq} item 1) and $\seq{x}$ is bounded, and $\rho_k\to\infty$, we obtain 

\begin{equation}\label{eqn:eq3}
\lim_{k\to\infty}\|\proj{(\K^k)\pol}{g(x^k)}\|=0,
\end{equation}
which means $g(w)\in \K$ because $$\|\ppol{g(x^k)}\|\leqslant \|g(x^k)-\proj{\K^k}{g(x^k)}\|+\|\proj{\K^k}{g(x^k)}-\pcon{g(x^k)}\|,$$ given that both terms vanish due to \eqref{eqn:eq3} and Lemma \ref{lem:projseq} item 2, respectively.

For every $z\in g^{-1}(\K)\cap B[\bar{x},\delta]$, there is a sequence $\seq{z}\to z$ such that $z^k\in g^{-1}(\K^k)\cap B[\bar{x},\delta]$, so $$f(x^k)+(1/2)\|x^k-\bar{x}\|^2\leqslant P_k(x^k)\leqslant P_k(z^k)=f(z^k)+(1/2)\|z^k-\bar{x}\|^2.$$ Taking limits, we obtain $f(w)+(1/2)\|w-\bar{x}\|^2\leqslant f(z)+(1/2)\|z-\bar{x}\|^2$. Hence, $w$ is a global minimizer of the following localized problem:
\begin{equation}\label{prob:locncp}
\begin{array}{ll}
\underset{x\in \R^n}{\text{Minimize}} & f(x)+\frac{1}{2}\|x-\bar{x}\|^2\\
\text{subject to} & g(x)\in \K\\
&\|x-\bar{x}\|\leqslant \delta,
\end{array}
\tag{Loc-NCP}
\end{equation}
but the unique global minimizer of \eqref{prob:locncp} is $\bar{x}$, which means $w=\bar{x}$. For $k$ sufficiently large, we have $\|x^k-\bar{x}\|<\delta$ and, by the first-order conditions for \eqref{prob:penalty}, we obtain

\begin{equation}\label{eqn:eq4}
\nabla f(x^k)+Dg(x^k)^*[\rho_k\proj{(\K^k)\pol}{g(x^k)}]=-(x^k-\bar{x}).
\end{equation}
Defining $\mu^k\doteq \rho_k\proj{(\K^k)\pol}{g(x^k)}$ for every $k\in \N$ is enough to finish the proof taking into account the equality
$$\langle \mu^k, \proj{\K^k}{g(x^k)}\rangle = \rho_k\langle \proj{(\K^k)\pol}{g(x^k)}, \proj{\K^k}{g(x^k)}\rangle=0,$$ and \eqref{eqn:eq4} with $x^k\to \bar{x}$.

\end{proof}

It is not true that the continuity of $g$ together with the fact $\seq{\K}$ is a continuous approximation of $\K$ implies that $g^{-1}(\K)\subseteq \liminf g^{-1}(\K^k)$, even if all sets involved are closed convex cones. For example, take
\[
	g(x)\doteq \begin{pmatrix}
		x+1 \\ x^2
	\end{pmatrix}
	\quad \textnormal{and} \quad
	\K\doteq \textnormal{cone}((1,0))
	\quad \textnormal{and} \quad
	\K^k\doteq \textnormal{cone}((1,-1/k)), \ \forall k\in \N.
\]
In this case, we have $\K^k\to \K$, but $g^{-1}(\K)=\{(1,0)\}$ and $g^{-1}(\K^k)=\emptyset$ for every $k\in \N$. 

In order to ensure the validity of this inclusion, we will consider from now on that $\seq{\K}$ is an outer approximation of $\K$; that is, if $\K\subseteq \K^k$ for every $k\in\N$, then of course $g^{-1}(\K)\subseteq g^{-1}(\K^k)$ for every $k\in\N$, whence follows that $g^{-1}(\K)\subseteq \liminf_{k\in \N} g^{-1}(\K^k)$.
In Section~\ref{sec:approximation}, we will define an outer approximation $\K^k$ for the cone of copositive matrices $\K$ based on \cite{alper} to be used in the numerical experiments of Section~\ref{sec:numerical}. Since $\K\subseteq \K^k$, the hypothesis of local feasibility of Theorem~\ref{thm:ragpnec} is automatically satisfied.

\subsection{An augmented Lagrangian variant with projections onto $\K^k$}

Given a positive scalar sequence $\{\rho_k\}_{k\in \N}$, a continuous approximation $\seq{\K}$ of $\K$, a closed ball $\mathcal{B}\subseteq \mathbb{E}$ centered at the origin with radius $R$ such that $\mathcal{B}\cap(\K^k)\pol\neq\emptyset$ for every $k\in \N$ together with a sequence $\{\widehat{\mu}^k\}_{k\in\N}\subseteq \mathcal{B}$, let $L_{\rho_k,\widehat\mu^k}: \mathbb{R}^{n}\rightarrow \mathbb{R}$ be an \textit{augmented Lagrangian function with a cone approximation}, defined as 
    \begin{equation}
    L_{\rho_k,\widehat\mu^k}(x)\doteq f(x)+
    \frac{\rho_k}{2}
    \left\|\bproj{(\K^k)\pol}{g(x)+\frac{\widehat\mu^k}{\rho_k}}\right\|^{2}-\frac{1}{2\rho_k}\left\|\widehat\mu^k\right\|^{2}.
    \label{eqn:sal}
    \end{equation}

%\noindent
%\textcolor{red}{(Observation: (5) is equal to 
%\begin{eqnarray*}
%&=& f(x)+\frac{\rho_k}{2}\left\|g(x)+\frac{\widehat{\mu}^k}{\rho_k}-\bproj{K^k}{g(x)+\frac{\widehat{\mu}^k}{\rho_k}}\right\|^2-\frac{1}{2\rho_k}\|\widehat{\mu}^k\|^2\\
%&=& f(x)+ \left\langle\widehat\mu^k,g(x)-\bproj{K^k}{g(x)+\frac{\widehat{\mu}^k}{\rho_k}}\right\rangle + \frac{\rho_k}{2}\left\|g(x)-\bproj{K^k}{g(x)+\frac{\widehat{\mu}^k}{\rho_k}}\right\|^2.
%\end{eqnarray*}
%)}    

%Note that its partial derivative with respect to $x$ is given by
%        \begin{equation}\label{eqn:gradlag}
%    \nabla_x L_{\rho_k}(x, \mu)= \nabla f(x)+
%    Dg(x)^*\left(\rho_k\bproj{(\K^k)\pol}{g(x)+\frac{\mu}{\rho_k}}\right).
%    \end{equation}

Let us now consider %following algorithm:
Algorithm~\ref{alg:gframework}.

\begin{center}
	\begin{algorithm}[!h]
		\caption{General framework: Augmented Lagrangian}
		\label{alg:gframework}	
		
		\
		
		\textit{Inputs:} a sequence $\{\varepsilon_{k}\}_{k\in \N} \subseteq \mathbb{R}_{++}$ of scalars such that $\varepsilon_{k} \rightarrow 0$; real parameters $\tau>1$, $\sigma \in (0,1)$, $\rho_0>0$, and $R>0$; $v^{-1}\in \E$; and initial points $(x^{-1}, \widehat{\mu}^{0}) \in \mathbb{R}^{n}\times \mathcal{B}\cap (\K^0)\pol$. 
		
	For every $k\in \N$:
     \begin{enumerate}		
		\item Compute 
		some point $x^{k}$ such that 
		\begin{equation}
		\|  \nabla L_{\rho_k,\widehat\mu^k}(x^k) \|\leq \varepsilon_{k},\label{eqn:approxfeasibility}         
		\end{equation}
		using an iterative method starting from $x^{k-1}$;
		\item Update the multiplier
		  \begin{equation}\label{eqn:almult}
		    \mu^{k}\doteq \rho_{k}
		    \bproj{(\K^k)\pol}{g(x^k)+\frac{\widehat{\mu}^k}{\rho_k}}, 
          \end{equation}
          and define $\widehat{\mu}^{k+1}\doteq \frac{\min\{\|\mu^k\|,R\}}{\|\mu^k\|}\mu^k$ as the projection of $\mu^{k}$ onto $\mathcal{B}\cap (\K^k)\pol$, where $R$ is the radius of the ball $\mathcal{B}$;
          \item Define 
		\begin{equation}\label{eqn:alfeascompmeasure}
		v^{k}\doteq\frac{\widehat{\mu}^{k}}{\rho_{k}}-\bproj{(\K^k)\pol}{g(x^k)+\frac{\widehat{\mu}^k}{\rho_k}};       
		\end{equation}
		   
		 \item If $\|v^{k}\|\leq\sigma \|v^{k-1}\|$, 
			set $\rho_{k+1}\doteq\rho_{k}$. Otherwise, 
			choose some $\rho_{k+1}\geq \tau \rho_{k}$.  
		
		\
		    
     \end{enumerate}	
	\end{algorithm}
\end{center}

Note that Steps 3 and 4 imply that either $\rho_k\rightarrow \infty$ or there is some $k_0\in \N$ such that $\rho_k=\rho_{k_0}$ for every $k>k_0$ and $v^k\to 0$. With this in mind, we proceed by showing that Algorithm \ref{alg:gframework} generates sequences whose limit points satisfy R-AGP.

%Also, observe that $V^k$ is a joint measure of feasibility and complementarity with respect to $\K^k$, since $V^k=0$ if, and only if, $\widehat{\mu}^k=\Pi_{(\K^k)\pol}(\rho_k g(x^k)+ \widehat{\mu}^k)$, which in turn holds if, and only if, $\widehat{\mu}^k\in \K^k$, $g(x^k)\in \K^k$, and $\langle \widehat{\mu}^k,g(x^k)\rangle=0$ by the definition of projection.
  
\begin{theorem}\label{theo:fakkt}
Let $\bar{x}$ be a feasible limit point of a sequence 
$\seq{x}$ generated by Algorithm 1, for any given choice of parameters $\seq{\widehat\mu}$ and $\{\rho_k\}_{k\in \N}$ conforming to Steps 2 and 3. 
Then, $\bar{x}$ satisfies R-AGP.
\end{theorem}
   
\begin{proof}
First, let us assume for simplicity that $x^k\to \xb$. Let $\seq{\mu}$ be as in~\eqref{eqn:almult} and observe that $\nabla L_{\rho_k,\widehat\mu^k}(x^k)=\nabla_{x} L(x^k, \mu^k)\to 0$. If $\rho_k\to \infty$, then
\[
	\langle \Pi_{\K^k}(g(x^k)),\mu^k\rangle = \rho_k\left\langle \Pi_{\K^k}(g(x^k))-\Pi_{\K^k}\left(g(x^k)+\frac{\widehat{\mu}^k}{\rho_k}\right),\Pi_{(\K^k)\pol}\left(g(x^k)+\frac{\widehat{\mu}^k}{\rho_k}\right)\right\rangle
\]
which implies (using Cauchy-Schwarz inequality and the nonexpansiveness of the projection) that
\[
	\begin{aligned}
		|\langle \Pi_{\K^k}(g(x^k)),\mu^k\rangle| & \leqslant \rho_k \left\| \Pi_{\K^k}(g(x^k))-\Pi_{\K^k}\left(g(x^k)+\frac{\widehat{\mu}^k}{\rho_k}\right) \right\| \left\| \Pi_{(\K^k)\pol}\left(g(x^k)+\frac{\widehat{\mu}^k}{\rho_k}\right) \right\| \\
		& \leqslant\rho_k\left\| \frac{\widehat{\mu}^k}{\rho_k} \right\| \left\| \Pi_{(\K^k)\pol}\left(g(x^k)+\frac{\widehat{\mu}^k}{\rho_k}\right) \right\|
	\end{aligned}
\]
but due to Lemma~\ref{lem:projseq} item 2, since $\rho_k\to \infty$ we know that $\|\Pi_{(\K^k)\pol}(g(x^k)+\widehat{\mu}^k/\rho_k)\|\to \|\Pi_{\K\pol}(g(\xb))\|=0$, so $\langle \Pi_{\K^k}(g(x^k)),\mu^k\rangle\to 0$. If $\rho_k=\rho_{k_0}$ for some $k_0$ and every $k\geq k_0$, then using the fact
\[
	\mu^k=\widehat{\mu}^k-\rho_k v^k
\]
we have that
\[
	\begin{aligned}
			\langle \Pi_{\K^k}(g(x^k)),\mu^k\rangle & = \langle \Pi_{\K^k}(g(x^k)),\widehat{\mu}^k \rangle - \langle  \Pi_{\K^k}(g(x^k)), \rho_k v^k\rangle,
	\end{aligned}
\]
but for any convergent subsequence of $\seq{\widehat{\mu}}$, which we will assume to be itself, so $\seq{\widehat{\mu}}\to \bar\mu$, 
%we have $\rho_k V^k\to \bar\mu-\Pi_{\K\pol}(\rho_{k_0} g(\xb)+\bar\mu)=0$. Then, 
we recall that $v^k\to 0$ and $\rho_k\to \rho_{k_0}$ to obtain that
\[
	\begin{aligned}
			\langle \Pi_{\K^k}(g(x^k)),\mu^k\rangle & \to \langle \Pi_{\K}(g(\xb)),\bar{\mu} \rangle=\langle g(\xb),\bar{\mu} \rangle=0,
	\end{aligned}
\]
which follows from $v^k\to 0$, since it implies $\bar{\mu}=\Pi_{\K\pol}(\rho_{k_0}g(\xb)+\bar\mu)$, which in turn holds if, and only if, $\bar\mu\in \K\pol$, $g(\xb)\in \K$, and $\langle g(\xb),\bar\mu \rangle=0$\footnote{In general, for any closed convex cone $\K\subset \E$ and any $y,z\in \E$, note that $z=\Pi_{\K\pol}(y+z)$ implies that $\Pi_{\K}(y+z)=y+z-\Pi_{\K\pol}(y+z)=y$.}. Thus, $\xb$ satisfies R-AGP.
\end{proof}

The proof above is an adaptation of the proof of \cite[Theorem 4.1]{Andreani2020}. Moreover, it is easy to show that if $\xb$ satisfies R-AGP and Robinson's CQ:
\[
	0\in\textnormal{int}(\Im(Dg(\xb)) + \K - g(\xb)),
\]
then $\xb$ must also satisfy the KKT conditions. This is a consequence of the boundedness of $\seq{\mu}$ in the presence of Robinson's CQ. Indeed, if $\seq{\mu}$ is not bounded, then we may take a subsequence such that $\|\mu^k\|\to \infty$, so
\[
	\lim_{k\to\infty}\frac{\nabla f(x^k)}{\|\mu^k\|}+Dg(x^k)^*\left[\frac{\mu^k}{\|\mu^k\|}\right]=\lim_{k\to \infty} Dg(x^k)^*\left[\frac{\mu^k}{\|\mu^k\|}\right]=0,
\]
but on the other hand we also have
\[
	\left\langle \Pi_{\K^k}(g(x^k)) , \frac{\mu^k}{\|\mu^k\|} \right\rangle\to 0 \quad \textnormal{and} \quad \Pi_{\K^k}(g(x^k))\to \Pi_{\K}(g(\xb))= g(\xb)
\]
by Lemma~\ref{lem:projseq} item 2. Define $H\doteq \Im(Dg(\xb)) + \K - g(\xb)$ and the above reasoning tells us that any limit point $\tilde\mu\neq 0$ of $\{\mu^k/\|\mu^k\|\}_{k\in \N}$ must belong to $H\pol= \Ker(Dg(\xb))\cap \K\pol\cap \{g(\xb)\}^{\perp}$. However, since $H\pol$ is a cone, this also means that $\alpha \tilde{\mu}\in H\pol$ for every $\alpha>0$. On the other hand, by Robinson's CQ, we also have that $\alpha \tilde{\mu}\in H$ for every $\alpha>0$ small enough, implying that $\tilde{\mu}=0$, which is a contradiction.

The proposed algorithm has some theoretical and numerical limitations. On the one hand, in case $\rho_k\rightarrow \infty$, the condition numbers of auxiliary problems would be unbounded. As a consequence, the numerical solution of auxiliary problems turn numerically unstable and difficult to deal with. This case may be expected when Robinson's CQ fails to hold.

On the other hand, for the interesting case of non polyhedral cones, a rapid local convergence can not be ensured, since the augmented Lagrangian does not include any information regarding the curvature of the cone. Nevertheless, the proposed algorithm provides a useful setting for the case of nonpolyhedral cones that are numericall difficult to represent and that lack of an efficient algorithm for calculating the projection.

\section{Continuous (polyhedral) approximation of the cone of co\-po\-si\-ti\-ve matrices}
\label{sec:approximation}

Now, we illustrate a continuous outer approximation of $\K$ (Definition~\ref{def:goodapprox}) by polyhedra through a
concrete example. Note that as discussed previously, this approximation satisfies the assumptions of Theorem \ref{thm:ragpnec}.

The cone of copositive matrices and its dual, the cone of
completely positive matrices,
are well studied closed convex cones. It is known that many NP-hard problems can be formulated as convex problems employing these cones \cite{duradapt,dursurvey,dursurvey2,shakedmonderer21}.
Among several inner and outer hierarchical and asymptotic approximations of
these cones, we consider the polyhedral approximation proposed by
Y{\i}ld{\i}r{\i}m \cite{alper}, due to its simplicity and cheap orthogonal projection onto.

For integers $m\geq 1$ and $r\geq 0$, consider the discrete set of vectors in $\R^m$ 
formed by regular grids of rational points on the
unit simplex:
\[
\delta^m_r\doteq\bigcup_{k=0}^r \{z\in \Delta_m \ : \ (k+2)z\in \N^m\},
\]
where $\Delta_m\doteq \{y\in\R^m \ : \ e^Ty=1\}$ is the $(m-1)$-dimensional unit simplex.
Let $\ell_r^m:=|\delta^m_r|$ be the number of elements in $\delta^m_r$, which can be roughly upper bounded
by $m^2(m^{r+1}-1)/(m-1)$ \cite{alper}.

Let $\SC^m$ be the set of $m\times m$ real symmetric matrices. 
For each $r=0,1,\ldots$, we define the following closed convex cones, which in fact are
polyhedra:
\[
\OC^m_r\doteq  \{Y\in \SC^m \ : \ d^TYd\geq 0, \forall d\in \delta^m_r\}
\]
and its dual
\begin{equation}
\left(\OC^m_r\right)^*\doteq \left\{\sum_{d\in\delta^m_r} \lambda_d dd^T \ : \ \lambda_d \geq 0\right\},
\label{eq:dualapprox}
\end{equation}
which give an outer approximation of the closed convex cone of copositive matrices
\[
\CC^m\doteq  \{Y\in \SC^m \ : \ u^TYu\geq 0, \forall u\in \R^m_+\}
\]
and an inner approximation of the closed convex cone of completely positive matrices
\[
\left(\CC^m\right)^* \doteq \left\{\sum_{i=1}^k (v_i)(v_i)^T \ : \ k\geq 1, v_i\in\R^m_+\right\},
\]
respectively.

It can be shown \cite{alper} that 
\[
(\OC^m_0)^*\subseteq(\OC^m_1)^*\subseteq \cdots \subseteq (\CC^m)^*
\subseteq \SC^m_+\cap \NC^m\subseteq \SC^m_+ \subseteq 
\SC^m_++\NC^m\subseteq
\CC^m \subseteq \cdots \subseteq
\OC^m_1 \subseteq \OC^m_0,
\]
$\textrm{cl}\left(\displaystyle\bigcup_{r\in \N}(\OC^m_r)^*\right)=(\CC^m)^*$, and
$\displaystyle\bigcap_{r \in \N}\OC^m_r=\CC^m$ for $\NC^m:=\{Y\in\SC^m \ : \
Y_{ij}\geq 0, \forall i,j=1,\dots, m\}$ where $\SC^m_+$ ($\SC^m_-$) is 
the cone of $m\times m$ positive (negative) semidefinite symmetric matrices. Moreover
$(\CC^m)^*=\SC^m_+\cap \NC^m$, $\CC^m=\SC^m_++\NC^m$ for $m\leq 4$, and $\{\OC^m_r\}_{r\in \N}$ is an outer continuous approximation of $\CC^m$ (Definition~\ref{def:goodapprox}).

An advantage of %{\st{employing these polyhedral approximations is that we can obtain an 
%orthogonal projection onto them relatively inexpensive. as follows. In our implementation, the set $\K^k$ is defined
%by a moderate number of linear inequalities and the computation of the projection onto is of m%oderate computation cost.}
our implementation based on the polyhedral approximations is that the sets $\K^k$ are defined by a moderate number of linear inequalities, such that the computational cost of projections is also moderate.
Namely, for $Y\in\SC^m$, let us compute its orthogonal projection onto the polar
cone of $\K^k:=\{Y\in\SC^m \ : \ d^T_iYd_i\geq 0, \ \forall d_i\in J^k\}\supseteq \OC^m_{\bar{r}}$ 
for some $\bar{r}\geq 0$ and $\emptyset \not=J^k\subseteq \delta^m_{\bar{r}}$.
Since $\Pi_{(\K^k)\pol}(Y)=-\Pi_{(\K^k)^*(-Y)}$, 
we need to solve the following optimization problem to
obtain the orthogonal projection $-Z$ of $Y$:
\[
\begin{array}{ll}\textrm{Minimize} & \|-Y-Z\|_F^2, \\
\textrm{subject to} & Z\in (\KC^k)^*. 
\end{array}
\]
Using the definition of $\KC^k$, the above problem is in fact
a convex quadratic program with $|J^k|$ non-negative variables:
\begin{equation}
\begin{array}{ll}\textrm{Minimize} & \lambda^TR\lambda+2s^T\lambda, \\
\textrm{subject to} & \lambda \in \R^{|J^k|}_+
\end{array}
\label{eqn:convexqp}
\end{equation}
for $R_{ij}=(d_i^Td_j)^2$ and $s_i=d_i^TYd_i$, where $d_i\in J^k$
for $(i,j=1,2,\ldots,|J^k|)$. This is due to the fact that
if we 
display all vectors of $J^k$ as rows of a larger matrix
\[
D\doteq \left(\begin{array}{c} d_1^T \\ d_2^T \\ \vdots \\ d^T_{|J^k|} \end{array}\right),
\]
the matrix $R$ will be the self Hadamard product of the Gram matrix $DD^T$, which 
is positive semidefinite. Thus, $R=(DD^T)\circ(DD^T)$ is also positive semidefinite.

As we can observe, other closed convex cones such as exponential, hyperbolicity \cite{Lourenco24} or any other cone can be treated in a similar manner as long as we have a continuous approximation of them by polyhedra.
This will permit us to minimize nonlinear functions over these cones by our proposed algorithm.

\section{Numerical Experiments}
\label{sec:numerical}

\subsection{Test problems}
\label{subsec:test}

As of our knowledge, there is no benchmark problems or reported computational results on minimizing nonlinear objective functions on difficult
convex cones such as the copositive one. 
Therefore, we created a set of 14  (NCP)  test problems in order to verify the performance of the proposed algorithm. Our test problems will have the (NCP) structure 
\begin{equation*}%\label{prob:ncp}
	\begin{aligned}
	& \underset{x\in \R^n}{\text{Minimize}}
	& & f(x), \\
	& \text{subject to}
	& & g(x)\in \K,\\
	\end{aligned}	
\end{equation*}
with $f(x)$ a selected nonlinear objective functions from \cite{minlplib,lue07,MGH81} (see Table~\ref{tab:function}) and $g(x)=Q_0+\sum_{i=1}^nx_iQ_i$ a randomly selected linear matrix conic constraint. The function $g(x)$ is taken linear just to simplify the random generation, but for the algorithm this is not a critical issue. In any case, the considered constraint $g(x)\in \K$ is not linear in its nature, and hard due to the difficulty already compressed in the cone $\K$.

\begin{table}[!htbp]
\caption{Nonlinear objective functions $f(x)$ for (NCP).}
\label{tab:function}
\begin{center}  
{\footnotesize
\begin{tabular}{|l|l|c|} \hline
\multicolumn{1}{|c|}{function's name} & \multicolumn{1}{c|}{function $f(x)$} & \multicolumn{1}{c|}{known local min. solutions} \\ \hline
convex quadratic (cq) & $x_1^2+x_2^2$ & $(0,0)$ \\ %#10
fractional convex (fc) & $x_1^2/(1+|x_1|)+x_2^2/(1+|x_2|)$ & $(0,0)$ \\ %#12
extended Rosenbrock (eR) \cite{MGH81}& $\displaystyle \sum_{i=1}^{n-1} (1-x_i)^2+100(x_{i+1}-x_i^2)^2$ & $(1,1,\ldots,1)$ \\ %#14
Freudenstein and Roth (FR) & $[-13+x_1+((5-x_2)x_2-2)x_2]^2$ & $(5,4)$ and  \\ %#102
\cite{MGH81}& \multicolumn{1}{r|}{$+[-29+x_1+((x_2+1)x_2-14)x_2]^2$} & $(11.41\ldots, -0.8968\ldots)$\\ 
Powell badly scaled (Pbs) \cite{MGH81}& $(10^4x_1x_2-1)^2+(e^{-x_1}+e^{-x_2}-1.0001)^2$ & $(1.098\ldots \cdot10^{-5},9.106\ldots)$ \\  %#103
Beale (B) \cite{MGH81}& $[1.5-x_1(1-x_2)]^2+[2.25-x_1(1-x_2^2)]^2$ & $(3,0.5)$ \\
& \multicolumn{1}{r|}{$+[2.625-x_1(1-x_2^3)]^2$} & \\  %#105
Powell singular (Ps) \cite{MGH81}& $(x_1+10x_2)^2+\sqrt{5}(x_3-x_4)^2+(x_2-2x_3)^4$ & $(0,0,0,0)$ \\
& \multicolumn{1}{r|}{$+10(x_1-x_4)^4$} & \\  %#113
Wood (W) \cite{MGH81}& $100(x_2-x_1^2)^2+(1-x_1)^2+90(x_4-x_3^2)^2$ & $(1,1,1,1)$ \\
& \multicolumn{1}{r|}{$+(1-x_3)^2+10(x_2+x_4-2)^2+(x_2-x_4)^2/10$} & \\  %#114
quartic polynomial (qp) & $\displaystyle\sum_{i=1}^n(x_i-1)^2+\left[\sum_{i=1}^ni(x_i-1)\right]^2+\left[\sum_{i=1}^ni(x_i-1)\right]^4$ & $(1,1,\ldots,1)$\\  %#136
Luenberger-Ye (LY) \cite{lue07}  & $x_1^2-5x_1x_2+x_2^4-25x_1-8x_2$ & $(20,3)$\\  %#137
ex4\_1\_5 \cite{minlplib} & $2x_1^2-1.05x_1^4+\frac{5}{30}x_1^6-x_1x_2+x_2^2$ & $(0,0)$ and \\  %#138
& & $\dag$ \ $(\pm 1.74755,\pm 0.97378)$ \\
ex8\_1\_4 \cite{minlplib} & $12x_1^2-6.3x_1^4+x_1^6-6x_1x_2+6x_2^2$ & $(0,0)$\\  %#140
ex8\_1\_5 \cite{minlplib} & $4x_1^2-2.1x_1^4+\frac{1}{3}x_1^6+x_1x_2-4x_2^2+4x_2^4$ & $(0,0)$ and \\  %#141
& & $(0.08984,-0.71266)$ \\
ex8\_1\_6 \cite{minlplib} & $[0.1+(x_1-4)+(x_2-4)^2]^{-1}$ & $(1.0004,1.0004)$ and \\
& \multicolumn{1}{r|}{$-[0.2+(x_1-1)^2+(x_2-1)^2]^{-1}$} & $(3.99995,3.99995)$  \\
& \multicolumn{1}{r|}{$-[0.2+(x_1-8)^2+(x_2-8)^2]^{-1}$} & \\  \hline %#142 \\
\multicolumn{3}{l}{\footnotesize{$\dag$ Not reported in \cite{minlplib}.}}
\end{tabular}}
\end{center}
\end{table}

We set $\K=\CC^m$ and generated $g(x)$ for a given nonlinear objective function  $f\colon \R^n\rightarrow \R$ by the following procedure:
\begin{itemize}
\item Fix a known local minimal solution $x^*\in\R^n$ of the objective $f(x)$ (see Table~\ref{tab:function}).
\item Choose a randomly generated vector
$\bar{x}\in\R^n$ where each component is from uniformly distributed samples on the interval $[10,100]$.
\item Select two positive semidefinite symmetric matrices, $P_1$ and $P_2$, whose eigenvalues are from uniformly distributed samples on the interval $[0,1]$.
\item Set the matrix variables $Q_0,Q_1,\ldots,Q_n\in\SC^m$ as a solution of the following semidefinite program
\[
\begin{array}{ll}
\displaystyle \textrm{Minimize}_{Q_0,\ldots,Q_n\in\SC^m}  & \displaystyle \textrm{tr}(Q_0+\sum_{i=1}^n\bar{x}_iQ_i), \\
\textrm{subject to} & \displaystyle Q_0+\sum_{i=1}^nx_i^*Q_i +P_1 \in \SC^m_-, \\
& \displaystyle Q_0+\sum_{i=1}^n\bar{x}_iQ_i-P_2 \in \SC^m_+. \\
\end{array}
\]
\end{itemize}

%We want to generate a problem which is feasible for $\bar{x}$, but infeasible for $x^*$. 
% For that, we first constructed two positive semidefinite matrices, $P_1$ and $P_2$, whose eigenvalues are from uniformly distributed sample on the interval $[0,1]$. Then we solve the following semidefinite program in the variables
% $Q_0,Q_1,\ldots,Q_n\in\SC^m$:
% \[
% \left\{\begin{array}{ll}
% \displaystyle \min_{Q_0,\ldots,Q_n\in\SC^m}  & \displaystyle \textrm{tr}(Q_0+\sum_{i=1}^n\bar{x}_iQ_i) \\
% \textrm{subject to} & \displaystyle Q_0+\sum_{i=1}^nx_i^*Q_i +P_1 \in \SC^m_- \\
% & \displaystyle Q_0+\sum_{i=1}^n\bar{x}_iQ_i-P_2 \in \SC^m_+. \\
% \end{array}\right.
% \]
The above procedure guarantees that $g(x) \doteq Q_0+\sum_{i=1}^nx_iQ_i$ is such that $g(x^*)$ and $g(\bar{x})$ have non-positive and non-negative eigenvalues, respectively. Therefore, $g(x^*)\notin \K$ and $g(\bar{x})\in \K$.
Cvxpy version~1.3.1 was used to solve the above
semidefinite program.

\subsection{Details of the implementation}

Algorithm~\ref{alg:gframework} computes an approximate local optimal solution when $\|\nabla L_{\rho_k,\widehat\mu^k}(x^k)\|$ and
$\|v^k\|$ are small enough since they correspond to the norms of derivatives of
the safeguarded augmented Lagrangian function $L_{\rho,\widehat\mu}(x)$ in 
relation to $x^k$ and $\widehat{\mu}^k$, respectively.
Therefore, $\seq{x}$ and $\seq{\mu}$ computed by Algorithm~\ref{alg:gframework} are sequences which define a limit point
$\bar{x}$ satisfying R-AGP according to Theorem~\ref{theo:fakkt}. 

Simple schemes were sought to define $\{\varepsilon_k\}_{k\in\N}$ and  $\seq{\KC}$ in Algorithm~\ref{alg:gframework}.
Among others, the empirical choice $\varepsilon_k\doteq\min\{\varepsilon_0,\|v^k\|_{\max}\}$ seemed to provide fast convergence,
where $\|v\|_{\max} = \max_{1\leq i,j \leq m} |v_{ij}|$. 
This strategy revealed to be superior than a conventional choice $\varepsilon_k\doteq \eta^k\varepsilon_0$ for $0 < \eta < 1$ 
which guarantees $\varepsilon_k \rightarrow 0$. Although our strategy may not guarantee $\varepsilon_k \rightarrow 0$, it 
appears to be reasonable when we expect that $\|v^k\|_{\max}\rightarrow 0$, whenever we succeed to compute $x^k$ satisfying (\ref{eqn:approxfeasibility}) at every iteration.

The strategy to choose $\seq{\K}$ has more variations and after some 
tests, we decided that the best strategy was to fix first the positive
integers $r_{\max}$, which defines $\OC^m_{r_{\max}}\supseteq \CC^m$ (see Section~\ref{sec:approximation}), and $\zeta<|\delta^m_{r_{\max}}|-|\delta^m_0|$. 
Then, we order all vectors in $\delta^m_{r_{\max}}$ such that 
the first $\ell^m_0=|\delta^m_0|$ elements are the vectors in $\delta^m_0$ (in any order), followed by the vectors
in $\delta^m_1\backslash\delta^m_0$ (in any order), until the vectors in $\delta^m_{r_{\max}}\backslash\delta^m_{r_{\max}-1}$
(in any order).
Set $J^0\doteq\delta^m_0$ in $\KC^k=\{Y \in\SC^m \ : \ d^T_iYd_i\geq 0, \ \forall d_i\in J^k\}$, and $J^k\doteq J^{k-1}\cup \{\textrm{first $\zeta$ vectors (whenever possible) in $\delta^m_{r_{\max}}$ which are not in $J^{k-1}$ in the above order}\}$. That is, at every iteration of Algorithm~\ref{alg:gframework}, $\K^k$ is defined by $|J^k|$ vectors and has exactly $\zeta$ more vectors than $|J^{k-1}|$, unless $0<|\delta^m_{r_{\max}}\backslash J^{k-1}|< \zeta$. In our numerical experiments (see Section~\ref{subsec:experiments}),
we considered test problems for $m=3$ and $r_{\max}=15$ which gives $\ell^m_{r_{\max}}=901$,
and $m=5$ and $r_{\max}=7$ which gives $\ell^m_{r_{\max}}=1816$. Notice that in order for our approximation of $\KC=\CC^m$ to be continuous as in Definition~\ref{def:goodapprox}, we would need to keep approximating the sets $\delta^m_r$ for $r>r_{\max}$, however, this is numerically intractable. %\textcolor{red}{I suggest removal: This strategy does not guarantee a continuous approximation of $\KC=\CC^m$ as in Definition~\ref{def:goodapprox} since
%$\KC\subsetneq \KC^k=\KC^p=\{y\in \SC^m \ : \ d_i^Tyd_i\leq 0, \forall d_i\in K^p=\delta^m_{r_{\max}}\}$ for all $k\geq p$ for 
%some $p$ and $\limsup_{k\in \N}\KC^k=\KC^p\supsetneq \K$.}

Another detail apparently hidden in the implementation of Algorithm~\ref{alg:gframework} is the radius $R$ of the
closed ball ${\mathcal B}$. % which we set as the ball of radius $R$, $\{x\in \R^n \ : \ \|x\|_{\infty} \leq R\}$. 
 At every iteration, we need to project $\mu^k$ onto ${\mathcal B}$ at step 2 and $R$ should be large enough to not bound the real
size of the Lagrange multiplier $\widehat{\mu}^{k+1}\in {\mathcal B}\cap (\KC^k)^{\pol}$. 
In the numerical experiments, we set 
$R=10^{12}$.

Finally, a scaling of the safeguarded augmented Lagrangian function $L_{\rho_k,\widehat{\mu}^k}(x)$ is effective, specially because some objective functions $f(x)$ are badly scaled. 
We divided the function
$L_{\rho_k,\widehat\mu^k}(x)$ (\ref{eqn:sal}) by
the average of the first five iterations of $\max\{1,\|\nabla L_{\rho_{k},\widehat{\mu}^{k}}(x^{k-1})\|_{\infty},\|\nabla f(x^{k-1})\|_{\infty}\}$.
% this statement is not completely correct since at the (-1)-th iteration, $x^{-1}$ and $\rho_0$, $\widehat{\mu}^0$ is used instead.

\subsection{Numerical results}
\label{subsec:experiments}
All numerical experiments were performed on Intel Core i7-10700 (2.90GHz, 8 cores) processor with 8GB of memory running python 3.9.16. Since the function (\ref{eqn:sal}) is only once differentiable \cite{Haeser2022}, scipy.optimize function with option ``method='BFGS' " was used to solve it with ``gtol=$\varepsilon_k$". This option was the best choice to have less failure
to compute an approximate $x^k$ satisfying (\ref{eqn:approxfeasibility}), which seems to be the achilles heel of the algorithm. The convex quadratic problem (\ref{eqn:convexqp}) on the other hand was solved by mosek 10.0.43.

 We adopted the following stopping criterion for Algorithm~1. The algorithm stops successfully if $\|\nabla L_{\rho_k,\widehat\mu^k}(x^k)\|\leq\varepsilon_L$, $\|v\|_{\max}\leq \varepsilon_V$, and $r=r_{\max}$. We forcefully stop the algorithm whenever we fail to compute (\ref{eqn:approxfeasibility}) in overall of 20\% of iterations, respecting a minimum of 14 iterations, since we observed that Algorithm~1 fails to converge whenever we cannot compute $x^k$ satisfying 
(\ref{eqn:approxfeasibility}) in some consecutive iterations.

The following parameter values were set for Algorithm~1: each coordenate of $x^{-1}\in\R^n$ randomly chosen from interval $[-100,100]$,
$v^{-1}=\frac{\widehat{\mu}^{0}}{\rho_{0}}-\bproj{(\K^0)\pol}{g(x^{-1})+\frac{\widehat{\mu}^0}{\rho_0}}$, $\widehat{\mu}^0=RI$, $R=10^{12}$, $\sigma=0.9$, $\tau=2.0$, $\varepsilon_V=\varepsilon_L=10^{-5}$; and $\rho_0=0.1$, $\varepsilon_0=1.0$ for $m=3$, and $\rho_0=1.0$, $\varepsilon_0=0.1$ for $m=5$, respectively. $I$ is the $m\times m$ identity matrix.
Choosing a large value $\widehat{\mu}^0$ increases numerical stability of the algorithm as this implies that $\|v^{-1}\|$ is small (see (\ref{eqn:alfeascompmeasure})).

Our first experiment aims to determine
the parameter $\zeta$, which is the number of vectors to add at every iteration to form $J^k$ from $J^{k-1}$ (see Section~\ref{sec:approximation}). 
Figures~\ref{fig:cm3-zeta} and~\ref{fig:cm5-zeta} show the performance profile in terms of wall-clock time when solving the 14 test problems 
described in Subsection~\ref{subsec:test} for matrix orders $m=3$ and $m=5$,
respectively. We can conclude that, in fact, there is a clear preference for the parameter $\zeta$ and it depends on the matrix sizes in this particular setting. 
Therefore, we fix $\zeta=45$ and $\zeta=70$ for the problems with matrix orders $m=3$ and $m=5$, respectively, in the main part of the 
numerical experiments.

%\begin{figure}[!htbp]
%\begin{center}   
%\includegraphics[scale=0.77]{m3cpuall.eps}
%\caption{Performance profile for $m=3$ when solving 14 problems by the ``proposed" method for different values of $\zeta$.}
%\label{fig:cm3-zeta-extra}
%\end{center}
%\end{figure}

\begin{figure}[!htbp]
\begin{center}   
\includegraphics[scale=0.43]{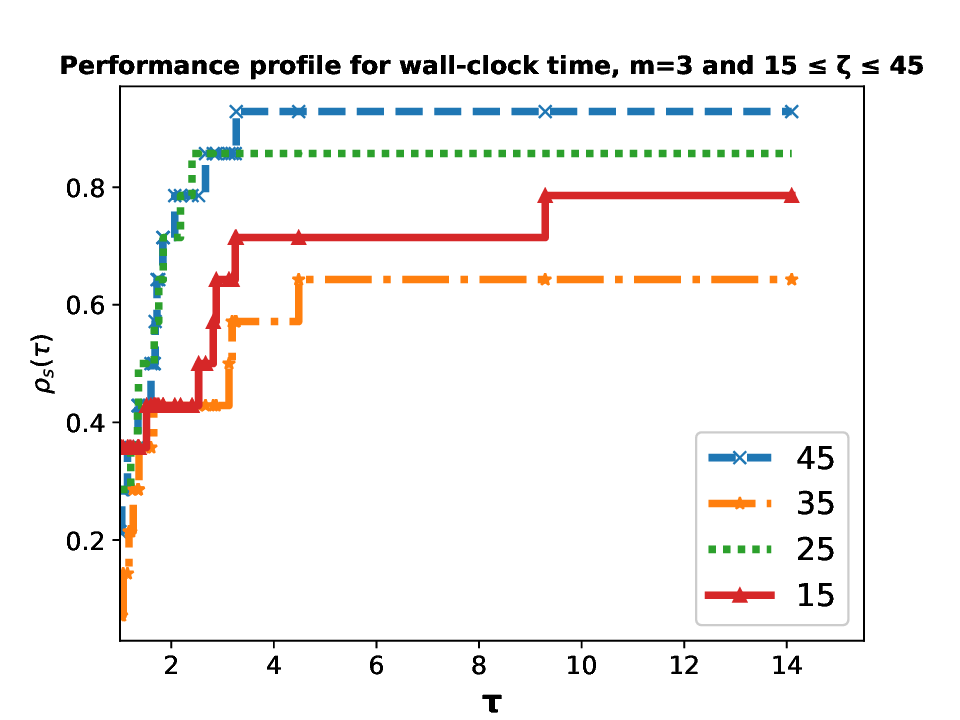}
\includegraphics[scale=0.43]{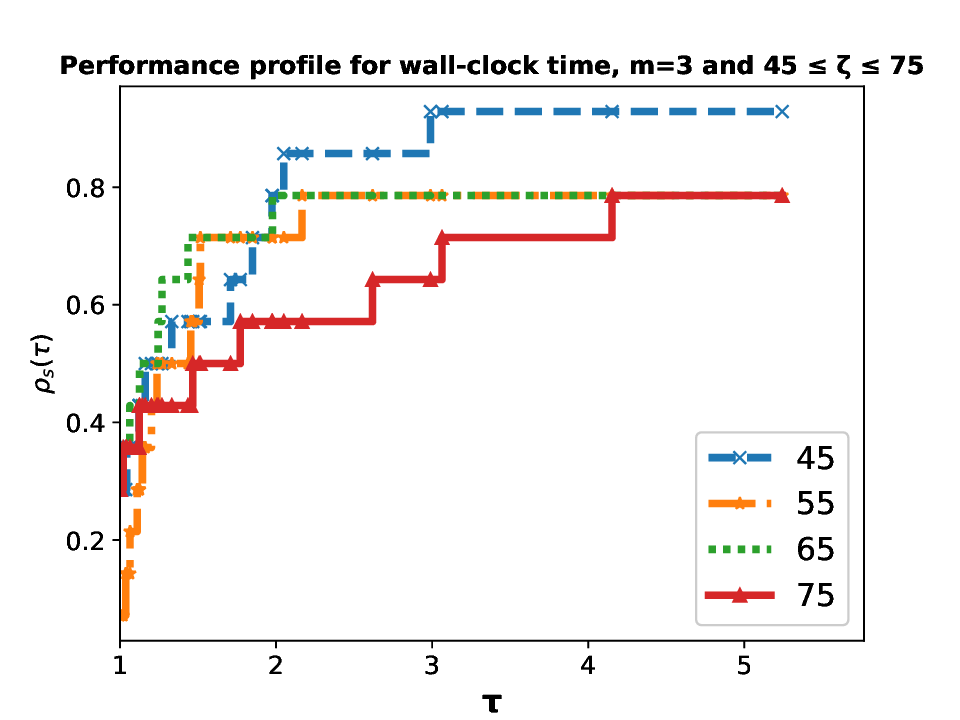}
\caption{Performance profile for $m=3$ when solving 14 problems by the ``proposed" method for values of $\zeta=15, 25, 35, 45, 55, 65,$ and $75$. It shows that $\zeta=45$ is the best choice.} 
\label{fig:cm3-zeta}
\end{center}
\end{figure}

%\begin{figure}[!htbp]
%\begin{center}   
%\includegraphics[scale=0.47]{m3_ct1.eps}
%\includegraphics[scale=0.47]{m3_ct2.eps}
%\caption{Performance profile for $m=3$ when %solving 14 problems by the ``proposed" %method for different values of $\zeta$.}
%\label{fig:cm3-zeta2}
%\end{center}
%\end{figure}

\begin{figure}[!htbp]
\begin{center}   
\includegraphics[scale=0.48]{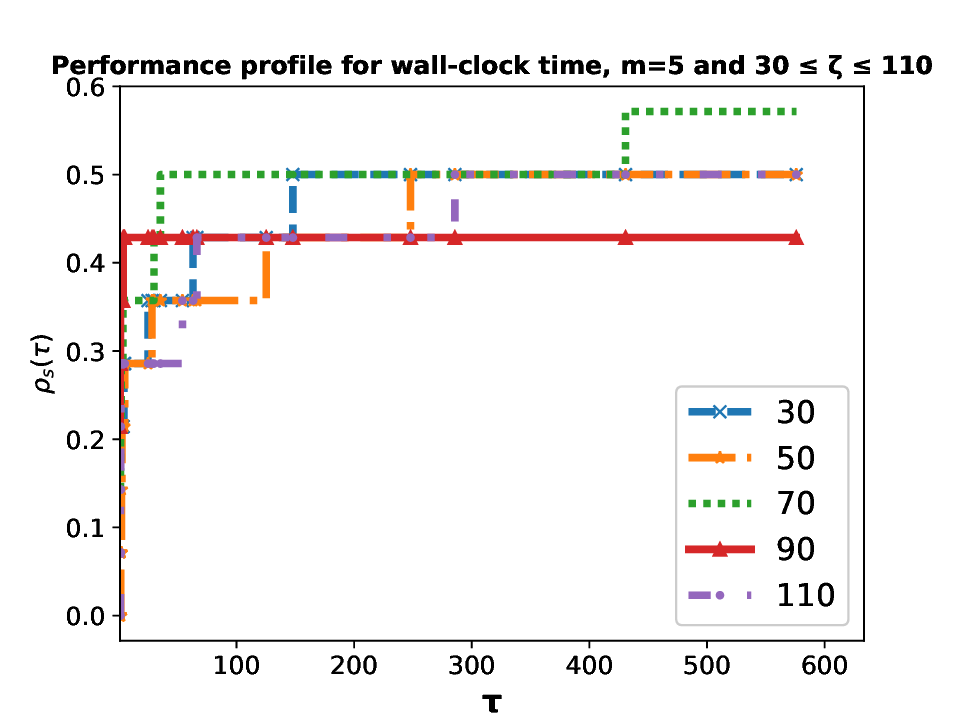}
\caption{Performance profile for $m=5$ when solving 14 problems by the ``proposed" method for values of $\zeta=30, 50, 70, 90,$ and $110$. It shows that $\zeta=70$ is preferred.}
\label{fig:cm5-zeta}
\end{center}
\end{figure}

In order to measure the effectiveness of our proposal in practice, we compared our 
algorithm to the standard augmented Lagrangian method applied to (NCP) where we just considered $J^0=J^1=\cdots=\delta^m_{r_{\max}}$. That is, the $\KC^k$ is fixed to the best approximation from the first iteration. We refer to this approach as ``standard" in our numerical results in contrast to ``proposed" for our proposal. 

%\textbf{COMMENTS:}
%\begin{itemize}
%    \item The results on both tables are identical. Is this OK?
%    \item Example $fc$ is not solved. Do we have a sound explanation for this?
%    \item tcpu-ime is closely related with the number of fails.  
%\end{itemize}

Table~\ref{tab:main1} 
details the numerical results for 14 problems when $m=3$ and $m=5$. 
The numbers of third and fourth columns are in boldface when they meet the stopping criterion as well as the column for $r$. The values of this column should be equal to $15$ (when $m=3$) and $7$ (when $m=5$), which correspond to the $r_{\max}$ fixed beforehand.
Observe that in problems such as ``fc", ``proposed" stops without attaining the pre-defined approximation for $\KC$, where $r$ should be 15 and 7 (instead of 14 and 5) due to excessive failure (4 out of 17 iterations and 10 out of 14 iterations) to satisfy (\ref{eqn:approxfeasibility}), respectively.
The column ``fails" indicates the number of times (\ref{eqn:approxfeasibility}) is not met for the current $\varepsilon_k$ among all iterations at ``it."

We also computed 
$\min_{Y\in\SC^m_++\NC^n} \|g(x^k)-Y\|_F$ that gave values between $3.6$e-$15$ to $6.9$e-$02$, omitted from the table. If this value is close to zero, we can understand that $g(x^k)\in \KC=\CC^m$,
since $\SC^m_++\NC^m\subseteq\CC^m$, with equality holding for $m\leq 4$ (see Section~\ref{sec:approximation}). For instance for problem ``Pbs" with $m=3$ and ``proposed", we obtained this smallest value, indicating that $g(x^k)\in\CC^3$, while for problem ``LY" with $m=5$, both for ``proposed" and ``stardard", we obtained this largest value indicating that $g(x^k)\not\in\SC^5_++\NC^5$, but not certain whether $g(x^k)\in\CC^5$.
%Wallclock time
%and CPU time are reported, were the latter in the %cumulative time for the 8 cores, since in three cases %(``Pbs 
%($m=3$)", ``B ($m=3$)", and ``ex8\_1\_5 ($m=3$)") there is a slight difference of which method is faster. Some routines 
%as mosek 10.0.43 can fully explore multicore computation, while our implementation can not take advantage of it.

%\begin{landscape}
\begin{table}[!htbp]
\begin{center}  
\caption{Numerical results for 14 problems when solving by the proposed method (``proposed") and the ``standard" augmented Lagrangian method with order of matrices $m=3$ and $m=5$. It is considered solved whenever $\|\nabla L_{\rho_k,\widehat\mu^k}(x^k)\|\leq\varepsilon_L$, $\|v^k\|_{\max}\leq\varepsilon_V$ and $r=r_{\max}$ (which are in bold); $r_{\max}=15$ for $m=3$ and $r_{\max}=7$ for $m=5$; ``fails" means \# of iterations (\ref{eqn:approxfeasibility}) was not satisfied.}
\label{tab:main1}
{\footnotesize
\begin{tabular}{|c|c|r|r|r|r|r|r|r|} \hline
problem & strategy & $\|\nabla L_{\rho_k,\widehat\mu^k}(x^k)\|$ & $\|v^k\|_{\max}$  & $r$ & it. & fails & wall time (s) & time/it. (s) \\ \hline
cq ($m=3$) & proposed & {\bf 4.197e-06} & {\bf 3.590e-07} & {\bf 15} & 30 & 2 & 148.17 & 4.78 \\
& standard & {\bf 4.203e-06} & {\bf 3.928e-06} & {\bf 15} & 29 & 1 & 170.45 & 5.68 \\ \cline{2-9}
cq ($m=5$) & proposed & {\bf 6.079e-07} & {\bf 3.702e-06} & {\bf 7} & 31 & 5 & 1621.38 & 50.67 \\
& standard & {\bf 6.836e-07} & {\bf 3.141e-06} & {\bf 7} & 26 & 1 & 1251.94 & 46.37 \\ \hline
fc ($m=3$) & proposed  & 1.272e-02 & 4.045e-03 &  14 & 17 & 4 & 94.74 & 5.26 \\
& standard & 1.225e-02 & 4.155e-03 &  {\bf 15} & 17 & 4 & 221.83 & 12.32 \\  \cline{2-9} 
fc ($m=5$) & proposed & 4.005e-02 &4.167e-04 &  5 &14 & 10& 170.68 & 11.38 \\
&standard & 9.934e-03 & 5.303e-02 &  {\bf 7} & 14 & 9 & 2920.29 & 194.69 \\ \hline
eR ($m=3$)& proposed & {\bf 1.432e-07} & {\bf 1.107e-08} &  {\bf 15} & 45 & 1 & 235.02 & 5.11 \\
$n=5$ & standard & {\bf 3.528e-08} & {\bf 9.324e-11} &  {\bf 15} & 43 & 0 & 201.86 & 4.59 \\ \cline{2-9}
eR ($m=5$)& proposed & {\bf 3.068e-07} & {\bf 7.386e-09} &  {\bf 7} &36&
1& 2763.04 & 74.68 \\
$n=5$ & standard & {\bf 9.626e-06} & {\bf 2.278e-07} &  {\bf 7} & 35 & 0 & 1663.28  & 46.20 \\ \hline
FR ($m=3$) & proposed & {\bf 9.131e-06} & {\bf 7.198e-07} & {\bf 15} & 37 & 0 & 56.19 & 1.48  \\
&standard & {\bf 1.323e-08} & {\bf 9.427e-09} &  {\bf 15}  & 37 & 0 & 150.44 & 3.96 \\ \cline{2-9}
FR ($m=5$) & proposed & {\bf 3.925e-06} & {\bf 4.340e-07} &  {\bf 7} & 33 & 1 & 258.59 & 7.61 \\
& standard & {\bf 4.008e-06} & {\bf 1.848e-10} &  {\bf 7} & 33 & 0 & 1036.54 & 30.49 \\ \hline
Pbs ($m=3$)& proposed &  {\bf 2.467e-15} & {\bf 2.014e-10} &  {\bf 15} & 20 & 0 & 16.04 & 0.76\\
&standard & {\bf 1.779e-10} & {\bf 3.332e-10} &  {\bf 15} & 1 & 0 & 28.27 & 14.14 \\ \cline{2-9}
Pbs ($m=5$)& proposed & {\bf 9.794e-06} & {\bf 1.352e-06} &  {\bf 7} &  48 & 2&  1369.89 & 27.96 \\
& standard & 1.429e-05 & {\bf 5.431e-06} &  {\bf 7} & 60 & 12 & 4753.46 & 77.93 \\ \hline
B ($m=3$)& proposed & {\bf 5.611e-08} & {\bf 8.549e-08} &{\bf 15} & 79 & 0 & 228.36 & 2.85 \\
&standard & {\bf 2.299e-09} & {\bf 8.696e-07} & {\bf 15} & 79 & 0 & 237.34 & 2.97 \\ \cline{2-9}
B ($m=5$)& proposed & 1.129e-03 & 2.099e-03 &  {\bf 7} &  84 & 17 & 8708.77 & 102.46 \\
&standard & 2.757e-03 & 2.714e-04 & {\bf 7} & 82 & 17 & 8407.61 & 101.30 \\ \hline
Ps ($m=3$)& proposed & {\bf 7.449e-06} & {\bf 3.722e-06} & {\bf 15} & 43 & 5 & 617.06 & 14.02 \\
& standard & 1.964e-04 & {\bf 2.181e-06} &  {\bf 15} & 55 & 11 & 693.93 & 12.39 \\ \cline{2-9}
Ps ($m=5$)&proposed  & {\bf 8.866e-06}  & 1.422e-05 &  {\bf  7} &  37
& 8 & 2787.54 & 73.36 \\ 
&standard & 1.111e-05 & 1.584e-05 & {\bf 7} & 40 & 8 & 4178.89 & 101.92 \\ \hline
W ($m=3$) & proposed & {\bf 9.936e-06} & {\bf 5.948e-06} & {\bf 15} & 45 & 1 & 345.71 & 7.52 \\
&standard & {\bf 8.095e-06} & {\bf 8.633e-06}  & {\bf 15} & 43& 2 & 274.10 & 6.23 \\ \cline{2-9}
W ($m=5$) & proposed & 2.418e-04 &1.629e-04 &  {\bf  7} & 49 &
10 &  6986.63 & 139.73 \\ 
&standard  & 5.245e-05 & 8.211e-05 &  {\bf 7} & 50 & 10 & 7367.50 & 144.46 \\ \hline
qp ($m=3$) &proposed & {\bf 9.756e-06} & {\bf 9.944e-06}& {\bf 15} & 34 & 1 & 134.30  & 3.84 \\
$n=5$ &standard & 1.210e-05 & {\bf 1.464e-06}  & {\bf 15} & 47 & 10 & 608.66 & 12.68 \\ \cline{2-9}
qp ($m=5$) & proposed & {\bf 2.375e-07} &  {\bf 9.955e-07} &  {\bf 7} &  34 & 
1 & 1098.84 & 31.40 \\
$n=5$ & standard & {\bf 1.564e-06} & {\bf 9.912e-09}  & {\bf 7} & 34 & 0 & 959.36 & 27.41 \\ \hline
LY ($m=3)$ & proposed & {\bf 9.260e-07} & {\bf 4.750e-06}  & {\bf 15} & 37 & 3 & 206.93 & 5.45 \\
& standard & {\bf 1.892e-06} & {\bf 2.072e-06} &  {\bf 15} & 40 & 5 & 437.75 & 10.68 \\ \cline{2-9}
LY ($m=5$) & proposed  & 3.497e-04 & 1.375e-05 & {\bf 7} & 39 & 8 & 3110.35 & 77.76 \\
& standard & 2.615e-05 & {\bf 3.067e-06} & {\bf 7} & 39 & 8 & 3712.36 & 92.81 \\ \hline
ex4\_1\_5 ($m=3)$ & proposed & {\bf 4.855e-06} & {\bf 1.588e-06} &   {\bf 15} & 34 & 4 & 246.38 & 7.04 \\
& standard & {\bf 8.450e-07} & {\bf 7.721e-06} &  {\bf 15} & 33 & 1 & 187.79 & 5.52 \\ \cline{2-9}
ex4\_1\_5 ($m=5$) & proposed &  {\bf 7.883e-07} & {\bf 4.096e-06} &  {\bf 7} & 26 & 0 & 394.90 & 14.63 \\
& standard & {\bf 5.657e-06} & {\bf 8.653e-07} &  {\bf 7} & 25 & 1 & 1456.68 & 56.03 \\ \hline
ex8\_1\_4 ($m=3$) & proposed & {\bf 5.602e-06} & {\bf 2.281e-06} &{\bf 15} & 46 & 6 & 367.15 & 7.81 \\
& standard & {\bf 7.500e-06} & {\bf 5.508e-06} &  {\bf 15} & 50 & 10 & 552.55 & 10.83  \\ \cline{2-9}
ex8\_1\_4 ($m=5$) & proposed & 8.105e-04 & 1.535e-04 &   {\bf 7} & 35 & 7 & 2393.85 & 66.50 \\
& standard & 1.346e-03 & 2.285e-05 &  {\bf 7} & 35 & 7 & 3237.19 & 89.92 \\ \hline
ex8\_1\_5 ($m=3)$ & proposed & {\bf 4.909e-06} & {\bf 2.280e-07} &  {\bf 15} & 31 & 1 & 151.47 & 4.73 \\
& standard & {\bf 3.230e-10} & {\bf 5.899e-06} &  {\bf 15} & 29 & 0 & 174.64 & 5.82 \\ \cline{2-9}
ex8\_1\_5 ($m=5$) & $\dag$proposed  & {\bf 4.751e-11} & {\bf 1.929e-08} & {\bf 7} & 26 & 1 & 570.25 & 21.12 \\
& standard & {\bf 3.946e-09} & {\bf 1.736e-06} &  {\bf 7} & 21 & 0 & 1336.28 & 60.74 \\ \hline
ex8\_1\_6 ($m=3)$ & proposed & {\bf 2.048e-08} & {\bf 7.997e-08} & {\bf 15} & 20 & 2 & 80.90 & 3.85 \\
& standard & {\bf 7.126e-09} & {\bf 5.140e-11} & {\bf 15} & 18 & 0 &  60.59 & 3.19 \\ \cline{2-9}
ex8\_1\_6 ($m=5$) & proposed  & {\bf 3.858e-08} & {\bf 7.980e-09} & 6 & 23 & 5 & 861.55 & 35.90  \\
& standard & {\bf 2.909e-08} & {\bf 1.500e-09} & {\bf 7} & 15 & 1 & 974.65 & 60.92 \\ \hline
\multicolumn{9}{l}{\footnotesize $\dag$ indicates mosek could not solve (\ref{eqn:convexqp}) properly in some iterations, but it did not affect the final results.}
\end{tabular}}
\end{center}
\end{table}

In general, the gradual polyhedral approximation of the closed convex cone $\KC=\CC^m$, which we are proposing, seems superior than considering a standard augmented Lagrangian method with fixed approximation $\KC^k$ from the beginning. This can be easily concluded from  
Figure~\ref{fig:cm}, which show the performance profile for the wall-clock time in Table~\ref{tab:main1}, when $m=3$
and $m=5$, respectively. Observe that some instances could not be solved at all by neither of methods, specially for the case $m=5$, showing that some nonlinear functions can be challenging for these type of algorithms.

In Table~\ref{tab:main1}, the last column gives the average wall-clock time per iteration,
which starts at iteration 0. These values are only reference values since each iteration
requires different amount of time, and more time is required when fail occurs 
due to increasing BFGS iterations to minimize (\ref{eqn:sal}). As we can observe, in general, ``standard" requires more time per iteration than ``proposed", because it needs to solve a larger problem with $J^0=\delta^m_{\max}$ from the first iterations, that is, larger problems to minimize in the projection (\ref{eqn:convexqp}). In the ``proposed" method, $J^0$ is set to $\delta^m_0$ and gradually increased.
There are few exceptions, ``B" for $m=5$, ``Ps" for $m=3$, and ``W" for $m=3$, even removing the
cases when ``proposed" has more failed iterations than ``standard". We believe that these are due to the averaging of the computational time since all these cases have
a higher number of iterations than other cases.

%Another point we can notice is that optimal values and optimal solutions obtained by ``proposed" and ``standard" can be different such as in ``eR", ``FR ($m=3$)", ``Pbs ($m=5$)", 
%``qp ($m=5$)", and ``ex8\_1\_5 ($m=3$)" because the objective function are polynomial of at
%least degree 3 or involve exponential functions and become quite sensitive to the values
%of variables.

\begin{figure}[!htbp]
\begin{center}   
\includegraphics[scale=0.43]{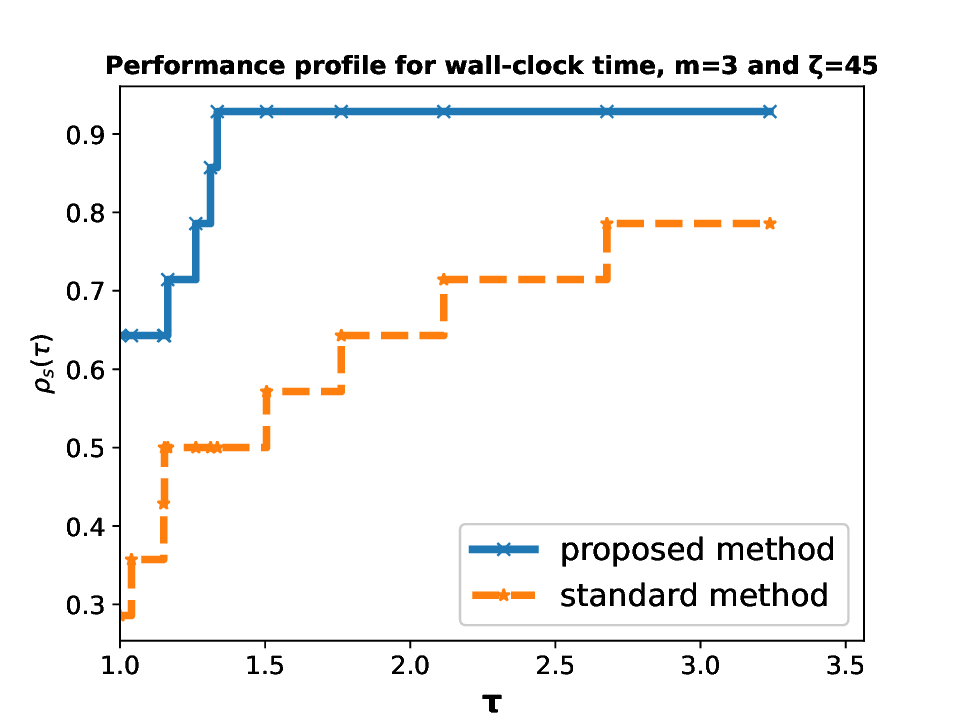}
\includegraphics[scale=0.43]{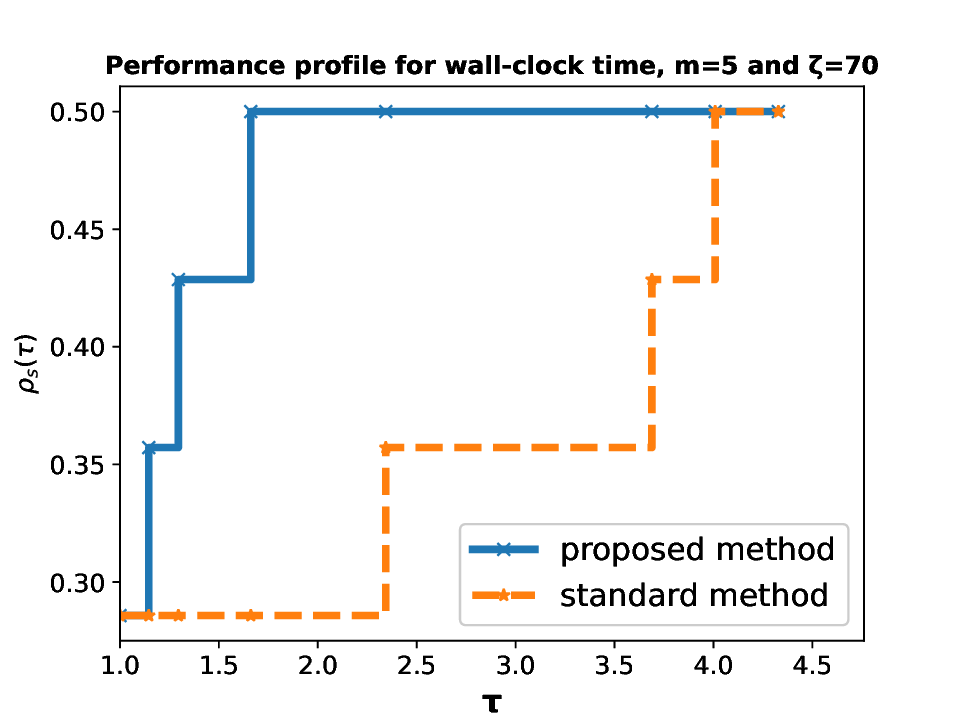}
\caption{Performance profile for $m=3$ (left) and $m=5$ (right) when solving 14 problems by the ``proposed" method and the ``standard" augmented Lagrangian method.}
\label{fig:cm}
\end{center}
\end{figure}

%\begin{figure}[!htbp]
%\begin{center}   
%\includegraphics[scale=0.47]{m3_45_wt.eps}
%\includegraphics[scale=0.47]{m3_45_ct.eps}
%\caption{Performance profile for $m=3$ %hen solving 14 problems by the %``proposed" method and the ``standard" %augmented Lagrangian method.}
%\label{fig:cm3}
%\end{center}
%\end{figure}

%\begin{figure}[!htbp]
%\begin{center}   
%\includegraphics[scale=0.47]{m5_70_wt.eps}
%\includegraphics[scale=0.47]{m5_70_ct.eps}
%\caption{Performance profile for $m=5$ %when solving 14 problems by the %``proposed" method and the ``standard" %augmented Lagrangian method.}
%\label{fig:cm5}
%\end{center}
%\end{figure}

%Finally, let us have a closer look at two instance of test problems.
%Figure~\ref{fig:Pbs5} show the behaviour of optimality measures along the
%iterations for ``Pbs ($m=5$)" and Figure~\ref{fig:qp3} for ``qp ($m=3$)"
%when solving by ``proposed" (left figures) and ``standard" (right figures),
%respectively. When comparing ``proposed" and ``standard", we observe that 
%in both cases that $\|\nabla L_{\rho_k,\widehat\mu^k}(x^k)\|$ and
%$\|V^k\|_{\max}$ start to decrease at almost same number of iterations
%due to the heuristics 
%\varepsilon^k\doteq\min\{\varepsilon_0,\|V^k\|_{\max}\}$. In %fact,
%when the minimum is attained at the right value and both methods are considering the best approximation of the cone $\CC^m$ since $r_{\max}$ is reached earlier. 
In conclusion, the ``proposed" method can save  computation time per iteration in the first iterations and a careful (and maybe conservative) update of the method 
seems to avoid the failure of satisfying $\|  \nabla L_{\rho_k,\widehat\mu^k}(x^k) \|\leq \varepsilon_{k}$ (\ref{eqn:approxfeasibility}) in later
iterations as it happens for the ``standard" method. %This can be thr reason that we obtained a

\section{Concluding Remarks}
\label{sec:conclusion}

The general optimization problem of minimizing a nonlinear function subject to nonlinear conic constraints has received increasing attention in the recent years. General algorithms for dealing with such generality are still under development. The linear case is of particular interest considering that in many applications the nonlinearities of the problem may be concentrated to lie in the cone itself. However, when it is not clear how to project onto the cone, a practical implementation is usually out of hand. In contrast, most algorithms require that one is able to deal with the full cone in every iteration of the method. In this paper we proposed an augmented Lagrangian algorithm that considers the possibility of iteratively approximating the cone in each iteration, without hindering well established global convergence results. Numerical experiments are conducted with the copositive cone and its polyhedral outer approximation where we demonstrate in a small collection of problems that our strategy is superior than the alternative one of considering the full cone in each iteration.

\section*{Acknowledgements}

This work was partially supported by grants 2018/24293-0, 2020/04585-7, 2020/07421-5, and 2023/08706-1 from the S\~ao Paulo Research Foundation (FAPESP) and grants 302000/2022-4 and 407147/2023-3 from CNPq. The authors are grateful for insightful comments by the reviewers.

\end{document}